\documentclass{amsart}

\RequirePackage[left,mathlines]{lineno}%
\RequirePackage{amsthm}%
\RequirePackage{amsmath}%
\RequirePackage{amsxtra}%
\RequirePackage{amstext}%
\RequirePackage{amssymb}%
\RequirePackage{latexsym}%
\RequirePackage{dsfont}%
\RequirePackage{ifthen}%

\ifx\pdfoutput\undefined%
   \IfFileExists{graphicx.sty}{\RequirePackage[dvips]{graphicx}
   \DeclareGraphicsExtensions{.eps,.ps}}{}
\else%
   \ifnum\pdfoutput=0
      \IfFileExists{graphicx.sty}{\RequirePackage[dvips]{graphicx}
      \DeclareGraphicsExtensions{.eps,.ps}}{}
   \else
      \IfFileExists{graphicx.sty}{\RequirePackage[pdftex]{graphicx}
      \DeclareGraphicsExtensions{.pdf,.png,.jpg}}{}
   \fi
\fi 

\newtheorem{theorem}{Theorem}[section]
\newtheorem{lemma}[theorem]{Lemma}

\theoremstyle{definition}

\theoremstyle{remark}

\theoremstyle{algorithm}
\newtheorem{algorithm}[theorem]{Algorithm}

\numberwithin{equation}{section}

\RequirePackage[colorlinks,citecolor=blue,urlcolor=blue]{hyperref}
\RequirePackage{color}%

\definecolor{red}       {rgb}{0.0,0.0,1.0}        
\definecolor{magenta}   {rgb}{0.0,0.0,1.0}        
\definecolor{cyan}      {rgb}{0.0,0.0,1.0}        
\definecolor{green}     {rgb}{0.0,0.4,0.3}        

\graphicspath{{./figs/}}%
\begin{document}



\title[Spectral projected gradient method]%
{A nonmonotone spectral projected gradient method for large-scale
topology optimization problems}


\author{R. Tavakoli}
\thanks{R. Tavakoli (corresponding author): Department of Material Science and Engineering, Sharif University of
Technology, Tehran, Iran, P.O. Box 11365-9466, email:
\href{mailto:tav@mehr.sharif.edu}{tav@mehr.sharif.edu}, URL:
\href{http://sites.google.com/site/rohtav/}{http://sites.google.com/site/rohtav/}.
}%

\author{H. Zhang}
\thanks{H. Zhang: Department of Mathematics, Louisiana State University, Baton Rouge,
LA, 70808, USA, email:
\href{mailto:hozhang@math.lsu.edu}{hozhang@math.lsu.edu}, URL:
\href{http://www.math.lsu.edu/~hozhang/}{http://www.math.lsu.edu/$\sim$hozhang/}.
}%


\date{\today}


\maketitle%



\begin{abstract}%

An efficient gradient-based method to solve the volume constrained
topology optimization problems is presented. Each iterate of this
algorithm is obtained by the projection of a Barzilai-Borwein step
onto the feasible set consisting of box and one linear constraints
(volume constraint). To ensure the global convergence, an adaptive
nonmonotone line search is performed along the direction that is
given by the current and projection point. The adaptive cyclic
reuse of the Barzilai-Borwein step is applied as the initial
stepsize. The minimum memory requirement, the guaranteed
convergence property, and almost only one function and gradient
evaluations per iteration make this new method very attractive
within common alternative methods to solve large-scale optimal
design problems. Efficiency and feasibility of the presented
method are supported by numerical experiments. \\




\end{abstract}%


\section{Introduction}%
\label{sec:int}%

The goal of topology optimization is to find optimal material
distribution in the given design domain subject to some
constraints governed by certain physical properties and/or some
other practical constraints during the design. In the past two
decades, advances in the theory of homogenization, optimization,
numerical analysis as well as newly developed engineering
approaches make topology optimization techniques to become a
standard tool of engineering design, in particular in the field of
structural mechanics. For more detailed literature review, one may
refer \cite{allaire2002soh,bendsoe2003tot} and the references
therein. In topology optimization, the design parameters are often
material properties (e.g., conductivity or stiffness tensor) and
the objective is often to minimize an integral functional defined
on the spatial domain with state variables satisfying a partial
differential equation (PDE) corresponding to certain physical law.
In addition, some bound constraints on the design variables and
usually a global material resource constraint are also often
enforced during the design.

One typical large class of topology optimization problems have the
structure that a nonlinear non-convex objective functional is
minimized over a feasible region defined by a second order
elliptic PDE together with bilateral bound and a single equality
constraints. In this paper, we focus on solving this class of
optimal design problems which have
 a wide range of applications in engineering design, such as compliance
mechanics, fluid dynamics, heat transfer, functionally graded and composite
materials \cite{allaire2002soh,bendsoe2003tot}.

In the nonlinear optimization literature, there are many
well-known optimization methods for finding solutions of general
optimization problem. Among them, the sequential quadratic
programming (SQP) \cite{gill2002ssa} method is widely used and
generally considered to be an efficient method for smooth
nonlinear optimization problems with constraints. Trust region
methods \cite{conn2000trm} are another class of well-studied
methods which have strong global and local convergence properties.
More recently, interior-point methods receives much more attention
because of their polynomial complexity. However, most of the above
mentioned methods require the evaluation or a certain type of
approximations of the hessian at each iteration. In addition,
these methods also often require to solve a linear system of
equations (usually indefinite, dense and is difficult to solve by
iterative methods) to a certain level at each iteration.
Therefore, when the problem size is very large, which often occurs
after the discretization of the topology optimization problem,
obtaining the Hessian information as well as solving very large
linear system of equations at each iteration could be very
expensive. Hence, although those second order methods enjoy fast
local convergence properties, they are generally not efficient to
solve very large-scale problems, especially when a solution with
very high accuracy is not strictly required.

On the other hand, some classes of optimization algorithms are
developed within engineering community to solve large-scale
engineering design problems. The first method of this kind is
known as CONLIN or convex linearization \cite{fleury1989ced}. More
advanced version of CONLIN, called method of moving
asymptotes (MMA), was introduced by Svanberg
\cite{svanberg1987mma} . Within each iteration of this method the
optimization problem is approximated by a convex separable
sub-problem, for which efficient solvers are available. The
globalization of the method is performed by either line-search
\cite{zillober1993gcv} procedures or conservative sub-iterations
\cite{svanberg2002cgc}.

Despite the improvements of these methods for general structural
optimization problems, a key issue to design an efficient
optimization method is to exploit the specific structure of the
desired problem. The optimality criteria (OC) method
\cite{bendsoe2003tot} is the first method developed
particularly to solve the resource constrained topology
optimization problems. Recently, OC becomes very popular and is
the most widely used method in the engineering community for such
problems. However, OC is not globally convergent and its
application is limited to some special problems like
thermal/structural compliance minimization
\cite[see:][ch.5]{allaire2002soh}. Moreover, the convergence rate
of OC is not very promising.

By applying the recent techniques developed in the field of
nonlinear optimization, in this paper we would like to design an
optimization algorithm for solving  very large-scale topology
optimization problems. Each iterate of this algorithm is obtained
by performing an adaptive nonmonotone line search along the line
segment connected by the current point and the projection point of
a cyclic Barzilai-Borwein (CBB) step onto the feasible set. Hence,
the method can be called the projected cyclic Barzilai-Borwein
(PCBB) method. Because of the special structure of the feasible
set of the problem, which consists of box constants and a single
linear constraint, it is possible to do the projection on the
feasible set very efficiently with linear time complexity. The
main attractive features of the presented algorithm are: producing
strictly feasible iterations; using almost one objective function
and gradient evaluations per iteration; O(n) memory
consumption (6n working memory); easily to be implemented; only
first order (gradient) information being required.


\section{Barzilai-Borwein methods}%
\label{sec:bb}%

Consider to solve the following finite-dimensional unconstrained
optimization problem,
\begin{equation}%
\label{eq:uo}%
    \min f({\bf x}), \quad {\bf x} \in {\mathbb R}^n,
\end{equation}
where $f: \mathbb{R}^n \rightarrow \mathbb{R}$ is continuously
differentiable, and ${\mathbb R}^n$ denotes the Euclidean space
with dimension $n$. Suppose that ${\bf x}_{0}$ is the starting point,
${\bf x}_{k}$ is the current point, and ${\bf g}_{k}$ is the
gradient of $f$ at ${\bf x}_{k}$, i.e., $\textbf{g}_k=\nabla
f(x_k)$. Then gradient methods generate the next iterative point by
\begin{equation}%
\label{eq:iter_uo}%
    \textbf{x}_{x+1} = \textbf{x}_k - \alpha_k \textbf{g}_k, \quad k = 0,1, \ldots, %
\end{equation}
where the stepsize $\alpha_k$ is computed by some line search techniques.
Two classical ways of selecting initial stepsize in the line searches
are given by the so called Steepest Descent (SD) and Minimal Gradient
(MG) methods, which minimize $f(\textbf{x}_k - \alpha \textbf{g}_k)$ and
 $\|\textbf{g}(\textbf{x}_k - \alpha \textbf{g}_k)\|$ along the search
direction $\textbf{g}_k$, respectively:
\begin{equation}%
\label{eq:sd stepsize}%
    \alpha_{k}^{SD} = \arg\min_{\alpha \in {\mathbb R}} f({\bf x}_{k} -
    \alpha {\bf g}_{k}),%
\end{equation}%
and
\begin{equation}%
\label{eq:mg stepsize}%
    \alpha_{k}^{MG} = \arg\min_{\alpha \in {\mathbb R}}
    \|\textbf{g}({\bf x}_{k} - \alpha {\bf g}_{k})\|.%
\end{equation}%
where $\| \cdot \|$ denotes the Euclidean norm of a vector. However,
it is well-known that SD and MG methods can be very slow when the
Hessian of $f$ is singular or nearly singular at the local minimum.
In this case the iterates could approach the minimum very slowly in
a zigzag fashion \cite{forsythe1968ads}.

The basic idea of Barzilai-Borwein (BB) \cite{barzilai1988tps}
method is to use the matrix ${\bf D}(\alpha_k) =
\frac{1}{\alpha_k}{\bf I}$, where ${\bf I}$ denotes the identity
matrix, to approximate of the Hessian $\nabla^2 f({\bf x}_k)$
by imposing a quasi-Newton condition on ${\bf D}(\alpha_k)$:%
\begin{equation}%
\label{eq:bb stepsize}%
    \alpha_{k}^{BB} = \arg\min_{ \alpha \in {\mathbb R}} \ \|
    {\bf D}(\alpha) \ {\bf s}_{k-1} - {\bf y}_{k-1}\|^2,%
\end{equation}%
where ${\bf s}_{k-1} = {\bf x}_k - {\bf x}_{k-1}$, ${\bf y}_{k-1}
= {\bf g}_k - {\bf g}_{k-1}$, and $k \geqslant 2$. By
straightforward calculation, BB stepsize obtained from
($\ref{eq:bb stepsize}$) is
\begin{equation}%
\label{eq:bb stepsize alpha}%
    \alpha_k^{BB} = %
    \frac{ {\bf s}_{k-1}^T {\bf s}_{k-1} } { {\bf s}_{k-1}^T {\bf y}_{k-1}}.
\end{equation}%
By symmetry, another alternative BB stepsize could be computed by:
\begin{equation}%
\label{eq:bb2 stepsize}%
    \alpha_{k}^{BB2} = \arg\min_{ \alpha \in {\mathbb R}} \ \|
    {\bf s}_{k-1} -  {{\bf D}^{-1}(\alpha)}\ {\bf y}_{k-1}\|^2,%
\end{equation}%
which gives
\begin{equation}%
\label{eq:bb2 stepsize alpha}%
    \alpha_k^{BB2} = %
    \frac{ {\bf s}_{k-1}^T {\bf y}_{k-1} } { {\bf y}_{k-1}^T {\bf y}_{k-1}}.
\end{equation}%
In contrast to the SD or MG methods in which the non-trivial and
expensive optimization problem (\ref{eq:sd stepsize}) or
(\ref{eq:mg stepsize}) need to be solved to obtain the initial
stepsize, the BB stepsize is readily available during the
iterations by formula (\ref{eq:bb stepsize alpha}) or (\ref{eq:bb2
stepsize alpha}). In practice, to keep the stability of the
numerical procedure, it is often to project the BB stepsize onto a
safeguard interval, that is to set
\begin{equation}%
\label{eq:bb_sg}%
    \bar{\alpha}_k^{BB} = \min \{ \alpha_{max},
    \max \{\alpha_{min}, \alpha_k^{BB} \}  \}
\end{equation}%
where $\alpha_{min}, \alpha_{max} \in \mathbb{R}$ and $ 0 <
\alpha_{min} << 1 << \alpha_{max} < \infty$.

The following Lemma shows the spectral property of
Barzilai-Borwein methods and it is due to this property that these
methods are usually called spectral gradient methods.

\begin{lemma}%
\label{lem:spectral}%
The Barzilai-Borwein stepsize, $\alpha_k^{BB}$, is the inverse of
the Rayleigh quotient, related to vector $s_{k-1}$, of the
averaged Hessian of the objective function between two consecutive
iterations $k-1$ and $k$.
\end{lemma}%
\begin{proof}%
By the Mean-Value Theorem and straightforward computations, one has%
\[
   \frac{\textbf{g}_k - \textbf{g}_{k-1}}{\textbf{x}_k -
   \textbf{x}_{k-1}} = %
   \int_0^1 \nabla^2 f (t \textbf{x}_k + [1-t] \textbf{x}_{k-1}) dt. %
\]
Therefore,
\[
   \frac{\textbf{s}_{k-1}^T\textbf{y}_{k-1}}{\textbf{s}_{k-1}^T\textbf{s}_{k-1}} =
   \frac{%
   \textbf{s}_{k-1}^T%
   \bigg(%
   \int_0^1 \nabla^2 f (\textbf{x}_{k-1} + t \textbf{s}_{k-1} ) dt
   \bigg)%
   \textbf{s}_{k-1}%
   } {\textbf{s}_{k-1}^T\textbf{s}_{k-1}},%
\]%
which complete the proof.
\end{proof}%
By Lemma \ref{lem:spectral}, one has $\Lambda_{min} \leqslant
(\alpha_k^{BB})^{-1} \leqslant \Lambda_{max}$, where
$\Lambda_{max}$ and $\Lambda_{min}$ are the maximum and minimum
eigenvalues of the averaged Hessian matrix $\int_0^1 \nabla^2
f(\textbf{x}_k+t \textbf{s}_k)\ d t$ respectively. Therefore,
$\alpha_k^{BB} \textbf{I}$ can be considered as an approximation
of the inverse of the averaged Hessian of the objective function.
This shows that the Barzilai-Borwein methods could incorporate
certain useful second order information with extremely little
additional expense of computational and memory cost compared with
SD or MG methods. Due to its easy implementation, efficiency and
low storage requirement, BB-type methods have been widely used in
many applications. Exceptional good performances of BB-type
methods have been observed for solving large-scale problems in
particular when only approximate (not very accurate) solutions are
desired, cf. \cite{fletcher2001bbm}.

It has been shown that if the {\it exact} steepest descent step
$(\ref{eq:sd stepsize})$ is reused in a cyclic fashion, the
convergence speed of gradient methods can be greatly accelerated.
However, it is often very expensive or impractical to find the
exact stepsize along the steepest descent searching direction for
large dimensional problems, unless the objective function is a
quadratic function. Hence, for non-quadratic objective functions,
it is unrealistic to apply some cyclic fashion of the steepest
descent method. On the contrary, the BB stepsize (\ref{eq:bb
stepsize alpha}) or (\ref{eq:bb2 stepsize alpha}) can still be
easily calculated. Hence, analogous to the cyclic steepest descent
method, the cyclic BB (CBB) has been introduced in
\cite{dai2006cbb} for general nonlinear optimization. The
numerical results given in \cite{dai2006cbb} show the CBB methods
have excellent numerical performances compared with SD methods and
even be competitive to some well-known nonlinear conjugate
gradient methods.
 Given an integer $m \geqslant 1$, which is the cycle
length, CBB stepsize can be expressed as
\begin{equation}%
\label{eq:cbb_m}%
    \alpha_{ml+i}^{CBB} = \alpha_{ml+1}^{BB}  \quad \mbox{ for } i = 1,
    \ldots, m, \quad l = 0, 1, \ldots,
\end{equation}%
The R-linear convergence of CBB method for a strongly convex quadratic
objective function has been proved in \cite{dai2003asg}, while the local
R-linear convergence for the CBB method at a local minimizer for general
nonlinear objective function has been established in \cite{dai2006cbb}.

\section{Projected Barzilai-Borwein methods}%
\label{sec:pbb}%

Consider the following constrained counterpart of problem (\ref{eq:uo})%
\begin{equation}%
\label{eq:co}%
    \min f({\bf x}), \quad {\bf x} \in \mathcal{D},
\end{equation}
where $\mathcal{D} \subseteq \mathbb{R}^n$ is a closed non-empty
convex set. Because of the constraints in (\ref{eq:co}), the
iterations generated by (\ref{eq:iter_uo}) may lie outside of the
feasible set $\mathcal{D}$. Therefore, (\ref{eq:iter_uo}) need to
be modified in order to maintain the feasibility of iterations.
Gradient projection methods (see: \cite{rosen1960gpm}) keep the
feasibility of iterates by frequently projecting trial steps
generated from (\ref{eq:iter_uo}) onto the feasible set.
Considering $\textbf{x}$ as the trial step, the projection point
$\textbf{y}$ of $\textbf{x}$ onto $\mathcal{D}$, denoted by
$\mathcal{P}_\mathcal{D}[\textbf{x}]$, can be computed by solving
the following minimization problem
\begin{equation}%
\label{eq:proj}%
    \textbf{y} := \mathcal{P}_\mathcal{D}[\textbf{x}] =
                  \arg \min_{\textbf{z} \in \mathcal{D}} \
                  \frac 12 {\| \textbf{x} - \textbf{z} \|}_2^2.%
\end{equation}
Since the feasible set is convex, problem (\ref{eq:proj}) always
has an unique solution. However, in general (\ref{eq:proj}) is
still a convex constrained quadratic programming problem, which
could be as difficult as the original problem. And for general
convex constrained large-scale problems, this projection step at
each iteration could be very time consuming and is normally the
most expensive part of gradient projection methods. Hence, there
is little interests on applying the gradient projection methods
for large-scale problem unless the gradient projection step can be
performed very efficiently. However, in some special cases when
efficient algorithms for calculating the projection
(\ref{eq:proj}) exist, for example there are only box or single
ball constraints, these gradient projection methods will be
attractive. Fortunately, as we will see in the next section, the
projection step can be performed very efficiently for the volume
constrained topology optimization problems. This makes it possible
for us to design gradient projection type algorithms for volume
constrained topology optimization.

Combining the Barzilai-Borwein stepsize rule and the gradient
projection method, the projected Barzilai-Borwein (PBB) method was
first introduced in \cite{birgin2000nsp}. In this method the
iteration updating formula (\ref{eq:iter_uo}) is modified to
\begin{equation}%
\label{eq:iter_co}%
    \textbf{x}_{x+1} = \textbf{x}_k + \beta_k \textbf{d}_k^\alpha, \quad k = 0,1, \ldots, %
\end{equation}
where $\beta \in \mathbb{R}_+$ and the search direction
$\textbf{d}_k^\alpha$ (descent direction) is computed by
connecting the current point to the projection point of the trial
iterate (\ref{eq:iter_uo}) based on the BB stepsize, that is
\begin{equation}%
\label{eq:pbb_descent}%
    \textbf{d}_k^\alpha =
    \mathcal{P}_\mathcal{D}[ \textbf{x}_k - \alpha_k^{BB} \textbf{g}_k ] - \textbf{x}_k. %
\end{equation}
In \cite{birgin2000nsp}, $\textbf{d}_k^\alpha$ was called
spectral projected gradient. It is not difficult to show that
$\textbf{d}_k^\alpha$ is a descent direction (see Lemma
\ref{lem:descent}). This together with the convexity of the
feasible set $\mathcal{D}$ would imply that for a sufficiently
small $\beta_k$, an iterate of (\ref{eq:iter_co}) will reduce the
objective function value while simultaneously preserve the
feasibility of the iterates.
\begin{lemma}%
\label{lem:descent}%
for all $\textbf{x}_k \in \mathcal{D}$ and $\alpha_k^{BB} > 0$,
\begin{enumerate}%
\item [(i)] $\langle \textbf{g}_k(\textbf{x}), \textbf{d}_k^\alpha(\textbf{x})\rangle
       \leqslant \frac {1}{\alpha_k^{BB}} \ {\| \textbf{d}_k^\alpha(\textbf{x}) \|}^2$.
\item [(ii)] $\textbf{d}_k^\alpha(\textbf{x}) = 0$ if and only if
$\textbf{x}$ is a stationary point for (\ref{eq:co}).
\end{enumerate}
\end{lemma}%
\begin{proof}%
see the Proposition 2.1 in \cite{hager2007nas}.
\end{proof}%

\section{Globalization by nonmonotone line search}%
\label{sec:ng}%

Using descent directions and simply applying the SD, MG or BB
stepsizes in gradient projection methods is generally not
sufficient to ensure global convergence of the iterates starting
from an arbitrary initial point.  Hence, to deal with general
nonlinear objective function, a globalization strategy is required
to guarantee the global convergence of the algorithm.

Monotonic gradient-based methods usually generate a sequence of
iterates for which a sufficient decrease in the objective function
(or the related merit function) is enforced at every iteration. In
many cases, the globalization strategy accepts the stepsize in the
search direction, if it satisfy the well-known Wolfe or Armijo
type conditions (cf. \cite[][ch. 3]{nocedal2006no}). This can be
accomplished using either of monotonic line searches or trust
region methods.

Since the search direction is parallel to the negative direction
of the gradient (projected gradient) in BB (PBB) methods, the
globalization of BB (PBB) methods by monotonic function value reduction often
reduces them to the classic SD (projected SD) method.
Hence, these monotonic methods will often distroy all of the advantages of BB
(PBB) methods in contrast to SD (projected SD) method. To maintain
the inherit spirits of BB-type methods, it is essential to accept
the initial BB-type stepsize as frequently as possible while
simultaneously ensure the global convergence. Hence, some
nonmonotone line search techniques need to be developed to
globalize the BB-type methods. The first nonmonotone line search
technique was developed in \cite{grippo1986nls} in which the main
goal was to accept the full Newton step as much as possible.
Combing the type of nonmonotone line search in
\cite{grippo1986nls}, the first globalized version of BB nethod
(GBB) was introduced in \cite{raydan1997bab}. Following
\cite{raydan1997bab}, the globalized PBB method (GPBB) was
suggested in \cite{birgin2000nsp}. In these methods the following
(weaker) objective function value decrease condition is enforced
during each iteration
\begin{equation}%
\label{eq:nm_cond}%
    f(\textbf{x}_{k+1}) \leqslant
    \max_{0 \leqslant j \leqslant m_k} f(\textbf{x}_{k-j}) +
    \delta \ \textbf{g}_k^T \textbf{d}_k, %
\end{equation}%
where $\delta \in (0,1)$, $\textbf{d}_k$ is the search direction and
$m_k$ is a nonnegative nondecreasing integer, bounded by some fixed
integer $M$. More precisely
\[
m_0 = 0 \quad \mbox{and } \quad  0 \leqslant m_k \leqslant
\min\{m_{k-1}+1, M \} \quad \mbox{for} \; k >0.%
\]
Based on the same motivations, the more efficient and adaptive
nonmonotone line searches were particularly designed for BB-type
methods in %
\cite{hager2007nas}.  %
%
 The globalized projected
cyclic Barzilai-Borwein (PCBB) algorithm based on these new
nonmonotone line search techniques can be described as follows:%

\begin{algorithm}%
\label{al:gpcbb}%
\begin{enumerate}%

\item []

\item []

\item [] Parameters:

\item [$\bullet$] $\epsilon \in [0, \infty)$, error tolerance.

\item [$\bullet$] $\delta \in (0, 1)$, descent parameter used in Armijo line search.

\item [$\bullet$] $\eta \in (0,1)$, decay factor for stepsize in Armijo line search.

\item [$\bullet$] $\alpha_{min}, \alpha_{max} \in (0,
\infty)$, safeguarding interval for BB stepsize.

\item []

\item [] Initialization:

\item [$\bullet$] $k=0, \textbf{x}_0=$ starting guess, and $f^r_{-1} =
f(\textbf{x}_0)$.

\item []

\item [] Main Loop: While $\| \mathcal{P}_\mathcal{D} [\textbf{x}_k - \textbf{g}_k] - \textbf{x}_k
\|_\infty > \epsilon$

\begin{enumerate}%

\item [1.] Choose $\bar{\alpha}_k \in [\alpha_{min},
\alpha_{max}]$.

\item [2.] Compute $d_k = \mathcal{P}_\mathcal{D} [\textbf{x}_k - \bar{\alpha}_k \textbf{g}_k] -
\textbf{x}_k$.

\item [3.] Choose $f_k^r$ such that $f(\textbf{x}_k) \leqslant f_k^r \leqslant \max\{ f_{k-1}^r,
f_k^{max}\}$ and $f_k^r \leqslant f_k^{max}$ infinitely often, where
$f_k^{max} = \max \{ f(\textbf{x}_{k-i}) : 0 \leqslant i \leqslant
\min (k, M-1)\}$.

\item [4.] Let $f^R$ be either $f_k^r$ or $\min\{ f_k^r,
f_k^{max}\}$.

\item [5.] Nonmonotone line search:

\begin{enumerate}%

\item [5.1.] If $f(\textbf{x}_k + \textbf{d}_k) \leqslant f^R + \delta  \ \textbf{g}_k^T
\textbf{d}_k$ then $\beta_k = 1$.

\item [5.2.] Else $\beta_k = \eta^j$, where $j>0$ is the smallest
integer such that $f(\textbf{x}_k + \eta^j \textbf{d}_k) \leqslant
f^R + \eta^j \delta \ \textbf{g}_k^T \textbf{d}_k$.

\end{enumerate}%

\item [6.] Set $\textbf{x}_{k+1} = \textbf{x}_k + \beta_k
\textbf{d}_k$ and $k = k + 1$.

\end{enumerate}%

\item [] End Main Loop.

\end{enumerate}%

\item []

\end{algorithm}%
The variable $f_k^r$ in Algorithm \ref{al:gpcbb} denotes the so
called ``reference" function value in nonmonotone line search. It
can be seen that the traditional monotone line search simply
corresponds to the choice of setting $f_k^r = f(\textbf{x}_k)$ at
each iteration. And the nonmonotone line search developed in
\cite{grippo1986nls} corresponds to the choice of setting $f_k^r
= f_k^{max}$. In our present study, $f_k^r$ is chosen based on
Algorithm \ref{al:ref_func} adapted from \cite{hager2007nas}. Let
$f_k$ denote $f(\textbf{x}_k)$. In the algorithm
\ref{al:ref_func}, the integer $a$ counts the number of
consecutive iterations for which $\beta_k=1$ in Algorithm
\ref{al:gpcbb} is accepted and the Armijo line search in step 5 is
skipped. The integer $l$ counts the number of iterations since the
function value is strictly decreased by an amount $\Delta>0$.

\begin{algorithm}%
\label{al:ref_func}%
\begin{enumerate}%

\item []

\item []

\item [R0.] If $k=0$, choose parameters: $A>L>0$, $\gamma_1, \gamma_2
>1$, and $\Delta>0$; initialize $a=l=0$ and $f_0^{min}=f_0^{max min} = f_0^r = f_{-1}^r =
f_0$.

\item [R1.] Update $f_k^r$ as follows:

    \begin{enumerate}%

    \item [R1.1.] If $l=L$, then set $l=0$, and

    \begin{equation}%
        f_k^r =\left\{%
        \begin{array}{ll}
        f_k^{max  min} & \texttt{if} \ \frac{f_k^{max} - f_k^{min}}{f_k^{max min} - f_k^{min}} \geqslant
        \gamma_1,\\
       f_k^{max} & \texttt{otherwise.}%
    \end{array}%
    \right.%
    \nonumber
    \end{equation}%

    \item []

    \item [R1.2.] Else If $a>A$, then set

    \begin{equation}%
        f_k^r =\left\{%
        \begin{array}{ll}
        f_k^{max} & \texttt{if} \ f_k^{max}> f_k \ \texttt{and} \ \frac{f_{k-1}^r - f_k}{f_k^{max} - f_k} \geqslant \gamma_2,\\
       f_{k-1}^r & \texttt{otherwise.}%
    \end{array}%
    \right.%
    \nonumber
    \end{equation}%

    \item []

    \item [R1.3.] Else set $f_k^r=f_{k-1}^r$.

    \end{enumerate}%

    \item [R2.] Set $f^R$ as follows in step 4 of Algorithm
    \ref{al:gpcbb}:

    \begin{enumerate}%

        \item [R2.1.] If $j=0$ (the first iterate in a CBB cycle)
        then $f^R = f_k^r$.

        \item [R2.2.] Else ($j>0$) $f^R = \min\{ f_k^r, f_k^{max}
        \}$

    \end{enumerate}%

    \item [R3.] If $\beta_k=1$ in Algorithm \ref{al:gpcbb} then
    $a=a+1$, Else ($\beta_k<1$) $a=0$.

    \item [R4.] If $f_{k+1}\leqslant f_k^{min} - \Delta$ then set
    $f_{k+1}^{max min}= f_{k+1}^{min} = f_{k+1}$ and $l=0$; Else set
    $l=l+1$, $f_{k+1}^{min}=f_k^{min}$ and $f_{k+1}^{max min} = \max\{ f_k^{max min},
    f_{k+1}\}$.

\end{enumerate}%

\end{algorithm}%

The variable $f_k^{max}$ in Algorithm \ref{al:ref_func} stores the
maximum of recent function values and $f_k^{min}$ stores the minimum
function value within the tolerance $\Delta$. The variable
$f_k^{max min}$ stores the maximum function value since the last new
minimum was recorded in $f_k^{min}$.

The condition $f(\textbf{x}_k) < f_k^r$ in step 3 of Algorithm
\ref{al:gpcbb} guarantees that the Armijo line search in step 5
can be satisfied. Notice that the requirement ``$f_k^r<f_k^{max}$
infinitely often" in step 3 which is required to ensure the global
convergence is a weaker condition. Besides Algorithm
\ref{al:ref_func} which satisfies this condition, this condition
can be satisfied by many other strategies. For example, it is
possible to set $f_k^r=f_k^{max}$ at every L iterations. In
Algorithm \ref{al:ref_func}, $f_k^r=f_k^{max}$ if
$f(\textbf{x}_{k-L}) - f(\textbf{x}_k) \leqslant \Delta$ for given
decrease parameter $\Delta>0$ and integer $L>0$.

Now lets give more details about computation of $\bar{\alpha}_k$
in step 1 of Algorithm \ref{al:gpcbb}. This parameter is computed
based on the safeguarded CBB scheme as follows. Let $j$ as an
integer that counts the number of times in which the current BB
step has been reused and let $m$ as the CBB memory in
(\ref{eq:cbb_m}), i.e., the maximum number of times the BB step
will be reused.

\begin{algorithm}%
\label{al:scbb}%
\begin{enumerate}%

\item []

\item []

\item [S0.] If $k=0$ choose $\bar{\alpha}_0 \in [\alpha_{min},
\alpha_{max}]$ and a parameter $\theta <1$ near 1; set $j=0$ and
flag=1. If $k>0$ set flag = 0.

\item [S1.] $0<|d_{ki}|<\bar{\alpha}_k |g_{ki}|$ for some $i$ (component of
vector) then set flag = 1.

\item [S2.] If $\beta_k= 1$ in Algorithm \ref{al:gpcbb} then set
j=j+1; Else ($\beta_k<1$) set flag =1.

\item [S3.] If $j \geqslant m$ or flag=1 or
    $\textbf{s}_k^T \textbf{y}_k / \| \textbf{s}_k \| \| \textbf{y}_k \|\geqslant
    \theta$ then:

    \begin{enumerate}%

    \item [S3.1.] If $\textbf{s}_k^T \textbf{y}_k \leqslant 0$ then

        \begin{enumerate}%

            \item [S3.1.1.] If $j>1.5m$ then set
                               $t = \min\{\|\textbf{x}_k\|_\infty,
                               1\}/\|\textbf{d}^1(\textbf{x}_k)\|_\infty$,\\
                               $\bar{\alpha}_{k+1}=
                                \min \{ \alpha_{max},
                                \max\{ \alpha_{min}, \beta_k\} \}$
                                and $j=0$; \\
                                where $\textbf{d}^1(\textbf{x}_k) =
                                \mathcal{P}_\mathcal{D}
                                [\textbf{x}_k - \textbf{g}_k]
                                - \textbf{x}_k$

            \item [S3.1.2.] Else $\bar{\alpha}_{k+1} = \bar{\alpha}_k$

        \end{enumerate}%

    \item [S3.2] Else set $\bar{\alpha}_{k+1}=
                          \min \{ \alpha_{max},
                          \max\{ \alpha_{min}, \alpha_k^{BB}\} \}$
                          and $j=0$.

    \end{enumerate}%

\end{enumerate}%

\end{algorithm}%

In Algorithm \ref{al:scbb}, the former BB stepsize is reused for
the current iterate unless one of the following conditions happens
in which the new BB stepsize is computed (see S3.2) : (I) the
previous BB stepsize is truncated by the projection step, i.e.,
when the trial point using BB stepsize lies outside of the
feasible domain and the gradient projection was performed (see
S1); (II) the previous BB stepsize is truncated by the line search
step (see S2 where $\beta_k<1$); (III) the number of times the BB
stepsize was reused reaches to its bound, i.e., $j\geqslant m$
(see S3); (IV) $\textbf{s}_k^T \textbf{y}_k / \| \textbf{s}_k \|
\| \textbf{y}_k \|$ is close to 1 (see \cite[][section
4]{dai2006cbb} for details about the justification for this
decision). The condition $\textbf{s}_k^T \textbf{y}_k < 0$ (see
S3.1) is equivalent to detection of the negative curvature in the
searching direction. Assuming that the objective function can be
well approximated by a quadratic function in the vicinity of the
current iterate, a relatively large stepsize should be used in the
next iteration (see S3.1.1) to reduce the function as much as
possible once a negative curvature is detected. This strategy is
similar to the original SPG algorithm (see \cite[][section
2]{birgin2001ass}).

Now lets briefly review the convergence theory of Algorithm
\ref{al:gpcbb}.

\begin{theorem}%
\label{th:gpcbb}%
Let $\mathcal{L}$ be the level set defined by%
\begin{center}%
$\mathcal{L} = \{ {\bf x} \in \mathcal{D} : f ({\bf x}) \le f({\bf
x}_0)$ \}.%
\end{center}%
We assume the following conditions hold:%
\begin{itemize}%
    \item[{\rm G1.}]%
        $f$ is bounded from below in $\mathcal{L}$ and $d_{\max} =
        {\sup}_k \|{\bf d}_k\| < \infty$.%
    \item[{\rm G2.}]%
        If $\bar{\mathcal{L}}$ is the collection of ${\bf x}\in \mathcal{D}$
        whose distance to $\mathcal{L}$ is at most $d_{\max}$, then
        $\nabla f$ is Lipschitz continuous on $\bar{\mathcal{L}}$.%
\end{itemize}%
Then either algorithm \ref{al:gpcbb} with $\epsilon=0$ terminates in
a finite number of iterations at a stationary point, or we have
$\liminf\limits_{k \to\infty} \| \textbf{d}^1(\textbf{x}_k)
\|_\infty = 0$.
\end{theorem}
\begin{proof}%
see the Theorem 2.2 in \cite{hager2007nas}.
\end{proof}%

When $f$ is a strongly convex function, (\ref{eq:co}) has a unique
minimizer ${\bf x}^*$ and the conclusion of the global convergence
Theorem \ref{th:gpcbb} can be strengthened as follows.

\begin{theorem}%
\label{th:gpcbb_convex}
Suppose $f$ is strongly convex and twice continuously differentiable
on $\mathcal{D}$, and there is a positive integer $L$ with the
property that for each $k$, there exists $j \in [k, k+L)$ such that
$f_j^r \le f_j^{\max}$. Then the iterates ${\bf x}_k$ of Algorithm
\ref{al:gpcbb} with $\epsilon=0$ converge to the global minimizer
${\bf x}^*$.
\end{theorem}
\begin{proof}%
see the Corollary 2.3 in \cite{hager2007nas}.
\end{proof}%

\begin{figure}[ht]%
\begin{center}%
\includegraphics[width=10.cm]{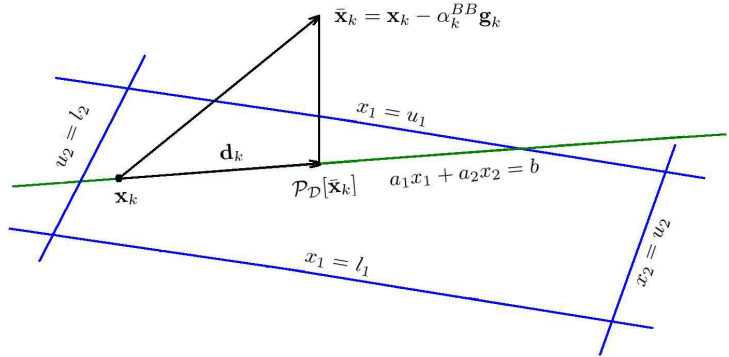}%
\caption{Projection onto the feasible set for $n=2$ related to the
volume constrained topology optimization problem; the feasible set
is a potion of line $\textbf{a}^T \textbf{x} = b$ which is located
inside the box $\textbf{l} \leqslant \textbf{x} \leqslant \textbf{u}$.}%
\label{fig:projection}%
\end{center}%
\end{figure}%

\section{Projection onto the feasible set}%
\label{sec:proj}%

The introduced algorithm in the previous section is general and can
be applied to any type of convex constrained optimization problems
which satisfy the requirements of Theorem \ref{th:gpcbb}. However,
the performance of the method is significantly affected by the efficiency
of the projection step. In this section, we explore the special structure of
the feasible set in the volume constrained topology optimization and
introduce a very efficient algorithm to do the projection step.

After discretization, the feasible set in the volume constrained topology
optimization problems has the following
form (see Figure \ref{fig:projection})%
\begin{equation}%
\label{eq:top_opt_feasible_set}%
   \mathcal{D} = \{ \textbf{x} \in \mathbb{R}^n : \textbf{a}^T \textbf{x} = b,\;  \textbf{l} \leqslant \textbf{x} \leqslant \textbf{u}
    \},
\end{equation}%
where $\textbf{a} \in \mathbb{R}^n$, $a_i \in \mathbb{R}_+$, $b
\in \mathbb{R}_+$, $\textbf{l}, \textbf{u} \in \mathbb{R}^n$, $0<
l_i \leqslant u_i < \infty$. Henceforth, in this section we denote
by $\mathcal{D}$ the feasible domain defined in
(\ref{eq:top_opt_feasible_set}).

Considering (\ref{eq:proj}) together with
(\ref{eq:top_opt_feasible_set}), for a given trial vector
$\textbf{x}$, $\mathcal{P}_\mathcal{D}[\textbf{x}]$ is the unique
minimizer of the following box constrained Lagrangian:%
\begin{equation}%
\label{eq:al_proj}%
\mathcal{L}(\textbf{z}; \lambda)= \frac 12 \|\textbf{z}\|_2^2 -
\textbf{x}^T \textbf{z} + \frac 12 \|\textbf{x}\|_2^2 + \lambda
(\textbf{a}^T \textbf{z} - b), \quad \textbf{z} \in \mathcal{B},
\end{equation}%
where $\lambda \in \mathbb{R}$ is a proper Lagrange multiplier
related to the volume constraint and the box constrained set
$\mathcal{B}$ is defined as $\mathcal{B} = \{\textbf{z}\in \mathbb{R}^n :
\textbf{l} \leqslant \textbf{z} \leqslant \textbf{u} \}$. For any
fixed value of $\lambda$, (\ref{eq:al_proj}) is a convex separable
minimization problem with respect to $\textbf{z}$, which has the
following explicit solution:
\begin{equation}%
\label{eq:al_proj_sol}%
    \textbf{z}(\lambda) = \max \{ \textbf{l},\  \min \{ \textbf{x}- \lambda \textbf{a}, \ \textbf{u}\}
    \},
\end{equation}
where the $\max$ and $\min$ operators are understood as
componentwise. Since solutions resulted from
(\ref{eq:al_proj_sol}) satisfy the bound constraints, to solve the
projection problem (\ref{eq:proj}) the remaining task then turns
out to find the Lagrange multiplier $\lambda^*$ such that
(\ref{eq:al_proj_sol}) satisfies the volume constraint. Hence, an
efficient algorithm is needed to find the (unique) root
$\lambda^*$ of the following nonlinear non-smooth one-dimensional
equation:
\begin{equation}%
\label{eq:dual_al}%
    g(\lambda) = \textbf{a}^T \textbf{z}(\lambda) - b = 0.
\end{equation}
Considering (\ref{eq:al_proj_sol}) together with (\ref{eq:dual_al}),
it can be seen that the graph of $g(\lambda)$ has $2n$ breakpoints at:%
\[
\lambda_i^l = (x_i- l_i)/a_i,  \quad \lambda_i^u = (x_i- u_i)/a_i,
\quad i=1, \ldots, n.
\]
Since $l_i \leqslant u_i$ and $a_i>0$ for all $i$, we have $\lambda_i^u
\leqslant \lambda_i^l$. So, each $z_i(\lambda)$ in
(\ref{eq:al_proj_sol}) can be expressed in the following form:%
\begin{equation}%
\label{eq:al_proj_sol2}%
    z_i(\lambda)=%
    \left\{%
\begin{array}{lll}
      u_i, & \texttt{if} & \lambda \leqslant \lambda_i^u,\\
      x_i - \lambda a_i,
      & \texttt{if} &  \lambda_i^u \leqslant \lambda  \leqslant \lambda_i^l,\\
      l_i,& \texttt{if} & \lambda \geqslant \lambda_i^l.
\end{array}%
\right.%
\end{equation}%
From (\ref{eq:al_proj_sol2}), it is obvious that for each $i$,
$z_i(\lambda)$ is a continuous piecewise linear and non-increasing
function of $\lambda$ (see Figure \ref{fig:x_i}). Therefore, by
$a_i>0$ for all $i$, $g(\lambda)$ is a continuous piecewise linear
and non-increasing function of $\lambda$. Then, with some
straitforward calculations, we can see that the root $\lambda^*$
of $g(\lambda)$ is unique and $\lambda^* \in [\lambda_{min},
\lambda_{max}]$ with $\lambda_{min} \leqslant 0 \leqslant
\lambda_{max}$, where $\lambda_{min} = \min\{\lambda_i^u\}$ and
$\lambda_{max} = \max\{\lambda_i^l\}$. So, starting from interval
$[\lambda_{min}, \lambda_{max}]$ and using some classical one
dimensional root finding method, it is possible to find
$\lambda^*$ up to the machine precision with a priori known
computational complexity. However, it is possible to use more
advanced root founding methods to improve the performance of this
step.

In our approach, the Brent's root finding method
\cite{brent1971agc} is employed to solve (\ref{eq:dual_al}). The
Brent's method does not assume the function differentiability and
is enable to manage the limited precision of the computed
arithmetic very well. In the worst condition, the convergence of
this method is never slower that of the bisection method. The
Brent's method has been proved to be very efficient and robust in
practice and it is currently accepted as a standard method for one
dimensional root finding problem (cf. \cite[][chapter
9]{press2007nre}). The specific implementation details of Brent's
method are available in the {\it{Numerical Recipe}} (see
\cite[][chapter 9]{press2007nre})..
\begin{figure}[ht]%
\begin{center}%
\includegraphics[width=7.cm]{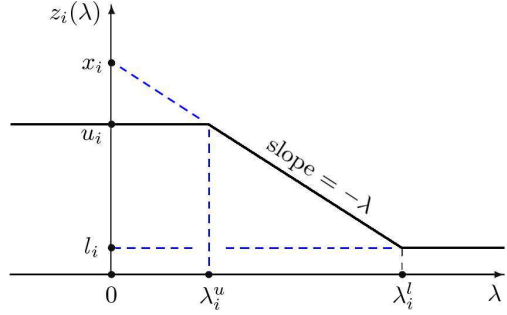}%
\caption{Plot of $z_i(\lambda)$ as a function of $\lambda$, cf. equation \ref{eq:al_proj_sol2}.}%
\label{fig:x_i}%
\end{center}%
\end{figure}%
%


\section{Numerical solution of topology optimization problems}%
\label{sec:topopt}%

To show efficiency of the proposed method for the volume
constrained topology optimization \cite{bendsoe2003topology}, in
this section we solve the standard model problems presented in
\cite{donoso2006numerical}. The description of the problem is as
follows: considering two isotropic conducting materials, with
thermal conductivities $k_\alpha$ and $k_\beta$, $0 < k_\alpha <
k_\beta$, in a simply connected design domain $\Omega \subset
\mathbb{R}^d \; (d = 2, 3)$, the goal is to mix these materials with
a fixed ratio in $\Omega$, such that the total temperature
gradient in $\Omega$ is minimized under the thermal load $f \in
L^2(\Omega)$. More specifically, the problem can be formulated as
\begin{equation}%
  \nonumber
    (P) \quad %
    \left\{%
\begin{array}{rll}
      \arg \min_{w}  J(w) =  &\frac 12 \int_{\Omega} \nabla \theta \cdot \nabla \theta \ d\textbf{x},
      &\\\\
      \texttt{subject to:} \\\\
      - \nabla \cdot (k(w) \nabla \theta) = & f(\mathbf{x})   &
    \texttt{in} \quad \Omega, \\
    \theta(\mathbf{x})  = & \theta_0(\mathbf{x})   &
    \texttt{on} \quad \partial\Omega,  \\ %
    k(w)    = & w^p k_\beta  + (1-w^p) k_\alpha             &
    \texttt{in} \quad \Omega,  \\ %
    \int_{\Omega} w \ d \textbf{x}   = & R |\Omega|, \quad 0<R<1, & 0  \leqslant w \leqslant
    1,
\end{array}%
\right.%
\end{equation}%
\\
where $\theta \in H^1(\Omega)$ is the state variable, $p\geqslant
1$ is a penalization factor and $w \in L^2(\Omega)$ is the control
parameter (topology indicator field). By some standard
derivations, the first order optimality condition of problem P can
be written as the following:

\begin{equation}%
  \nonumber
    (OC) \quad %
    \left\{%
\begin{array}{rll}
      - \nabla \cdot (k(w) \nabla \theta) = & f(\mathbf{x})   &
    \texttt{in} \quad \Omega, \\
    \theta(\mathbf{x})  = & \theta_0(\mathbf{x})   &
    \texttt{on} \quad \partial\Omega,  \\ %
    - \nabla \cdot (k(w) \nabla \eta) = & -\nabla\cdot(\nabla \theta)  &
    \texttt{in} \quad \Omega, \\
   \eta(\mathbf{x})  = & 0   &
  \texttt{on} \quad \partial\Omega, \\  %
    k(w)    = & w^p k_\beta  + (1-w^p) k_\alpha             &
    \texttt{in} \quad \Omega, \\  %
    \mathcal{P}_{\mathcal{D}}\big(G(\mathbf{x}) \big)    = & 0 &
    \texttt{in} \quad \Omega,  \\%
    G(\mathbf{x})    = & -p w^{p-1}(k_\beta - k_\alpha) \nabla \theta \cdot \nabla \eta &
    \texttt{in} \quad \Omega,  \\%
\end{array}%
\right.%
\end{equation}%
\\
where $\eta \in H^1_0(\Omega)$ is the adjoint state, $G$ is the
$L^2$ gradient of the objective functional with respect to $w$ and
$\mathcal{P}_{\mathcal{D}}(u)$ denotes the $L^2$ projection of
function $u$ onto the admissible set $\mathcal{D}$,
\[
\mathcal{D} = \{ w \in L^2(\Omega) \ | \  \int_\Omega
w(\mathbf{x}) d \mathbf{x}  = R |\Omega|, \quad 0<R<1, \quad 0
\leqslant w \leqslant 1 \}.
\]

By discretization of $\Omega$ into an $n$ control volumes, we have
the finite dimensional counterpart of problem (P). We also assume
that the state variable and the design parameter are defined at
the center of each control volume. Under these assumptions, the
admissible design domain $\mathcal{D}$ forms a simplex in
$\mathbb{R}^n$ which is identical to the continuous knapsack
constraints in (\ref{eq:top_opt_feasible_set}).  At each iterate
of the control parameter $w$, we solve the associated state PDE.
Then the discretized optimization problem would have the general
format of problem \eqref{eq:co}, which has a nonlinear objective
function with convex continuous knapsack constraints.

In our numerical experiments, the problem (P) was solved in two and
three dimensions for $\Omega = [0,1]^2$ and  $\Omega = [0,1]^3$.
The spatial domain is divided into $127^2$ and $31^3$
control volumes in two and three dimensions respectively. In all of
experiments, we set $k_\alpha = 1$, $f(\mathbf{x}) = 1$ and $R =
0.4$. And in these experiments two conductivity ratios 2 and 100
were tested, which is equivalent to setting $k_{\beta} = 2, 100$
respectively. The penalization factor $p$ is taken to be $1$ and
$10$ for conductivity ratio $2$ and $100$ respectively. The
governing PDE are solved by cell centered finite volume method using
central difference scheme and the related system of linear equations
are solved by a preconditioned conjugate gradient method with
relative convergence threshold $10^{-20}$. The optimization is
performed for $15$ iterations in these numerical experiments. Notice
that using finite volume method we do not observe any topological
instability phenomena  (the checkerboard pattern), which often
occurs by using finite element method for topology optimization
problems.

The input parameters related to optimization algorithms used in this
study are as follows: $\delta = 10^{-4}$,
$\eta = 0.5$, $\alpha_{min} = 10^{-30}$, $\alpha_{max} = 10^{30}$,
$A=40$, $L = 10$, $M=20$, $m=4$, $\gamma_1 = \gamma_2= 2$, $\theta = 0.975$.
Moreover, the initial value  $\alpha_0 $ for spectral step-size was
taken to be $ 1/\|\mathcal{P}_\mathcal{D} [\textbf{x}_k - \textbf{g}_k] - \textbf{x}_k
\|_\infty$. Another alternative choice for this parameter could be $\alpha_0
= 1/\| \mathcal{P}_\mathcal{D} [\textbf{x}_k - \textbf{g}_k] -
\textbf{x}_k \|_2$. However, for the testing problems in our numerical experiments the
former choice was considerably more efficient.

To evaluate the efficiency of the presented method, we compare our
results with the results obtained by the method of moving asymptotic
(MMA) \cite{svanberg1987mma}. The implementation of MMA available in
SCPIP code \cite{zillober2002scp}\footnote{The SCPIP code (in
Fortran) is freely available through personal request from its
original author (Christian Zillober:
christian.zillober(at)uni-wuerzburg.de).} is used in our
experiments. All default parameters are used in SCPIP code, except
the parameter for the constraint violation, for which the threshold
$10^{-7}$ is used in this study. That SCPIP code used in MMA
algorithm has two globalization strategies. The first one is
identical to that of the original MMA by Svanberg
\cite{svanberg2002cgc}, while the second one is globalization by
monotone line search method (see: \cite{zillober1993gcv}). The first
strategy is employed in our experiments, since our results with
second strategy was significantly worse in terms of the
computational cost. Note that in our procedure the PDE constraint is
solved upto the accuracy of the finite volume method we have applied
for solving this PDE, and the constraints
\eqref{eq:top_opt_feasible_set} in PCBB algorithm are satisfied upto
the machine precision.

The results of our numerical experiments including the variation of
the objective functional during the optimization process and the
final resulted topology ($w$-field) are shown in figures
\ref{fig:obj_hist_2d}, \ref{fig:obj_hist_3d}, \ref{fig:final_top_2d}
and \ref{fig:final_top_3d}. The plots in figures
\ref{fig:obj_hist_2d}, \ref{fig:obj_hist_3d} show the success of the
presented method to solve these topology optimization problems.
Roughly speaking, both methods behave very competitively in terms of
computational cost and final results. The main differences in the results
are related to the conductivity ratio 100, in which the PCBB
performs superior and its final objective function values are
considerably lower than that of MMA. The sign of the differences is
clear in the figures \ref{fig:final_top_2d} and \ref{fig:final_top_3d}
related to the final topologies. But for the
conductivity ratio equal to 2, it seems MMA behaves slightly better than PCBB.
However, the differences in terms of the objective
function values as well as the final topologies are almost negligible .

In all our numerical experiments, the final total number of function
and gradient evaluations was $15$ which is equal to the
number of optimization cycles. This result shows that both methods
used only one function and gradient evaluations per iteration in practice.
We believe that this is a key property for the success of MMA and makes
the method well accepted in the engineering design community.
In fact, MMA behaves very conservatively and uses small
steps to proceed toward a local minimum. More clearly, it does not
use (expensive) line search globalization strategy (unlike
alternative methods), but uses reasonably small steps such that
hopefully the merit function decreases sufficiently during each
iteration. Of course, whenever the merit function increases (or
sufficient decrease in merit function value violates), which
rarely happens in practice, it performs sub-cycles to ensure the desired monotonic
behavior. On the other hand, the nonmonotone PCBB methods enjoys such property
using an alternative strategy. In practice, by applying the nonmonotonic line
search, PCBB often uses only one function evaluation per optimization cycle.
As our results clearly show, this property makes the nonmonotone PCBB
method as a very competitive alternative to MMA for these class of
problems.

It is important to note that in all of our experiments presented
here, the projection step is used actively in PCBB method.
Therefore, the cyclic reuse of step size never employed in our testing problems
(cf. algorithm 4.3). Moreover, for the conductivity ratio 100, PCBB
method explored directions of negative curvature after a few
iterations, and so used large step-sizes which considerably
accelerate the convergence. This property plays an important role for the
superior results of PCBB compared with MMA in these cases
(cf. table \ref{tab:aplha}).

Besides the presented numerical results here, we have also applied
the presented method successfully to many families of topology
optimization problems with very satisfied results. But we omit the
details of these experiments here due to the limited space for this
paper.

\begin{figure}[ht]%
\begin{center}%
\includegraphics[width=13.cm]{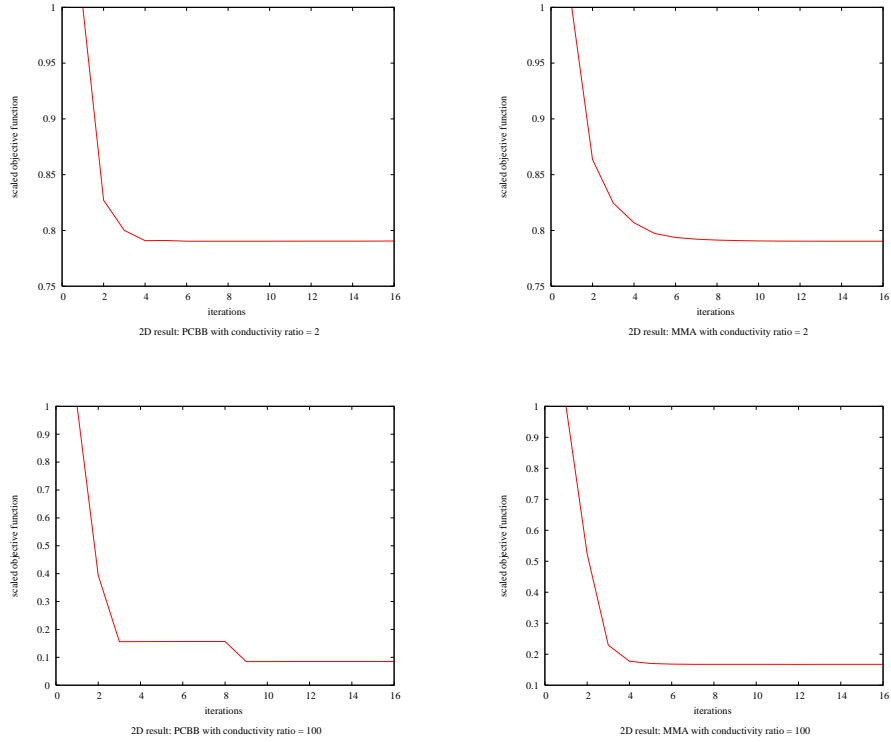}%
\caption{Variations of scaled objective function values during optimization
cycles for 2D examples.
}%
\label{fig:obj_hist_2d}%
\end{center}%
\end{figure}%
\begin{figure}[ht]%
\begin{center}%
\includegraphics[width=13.cm]{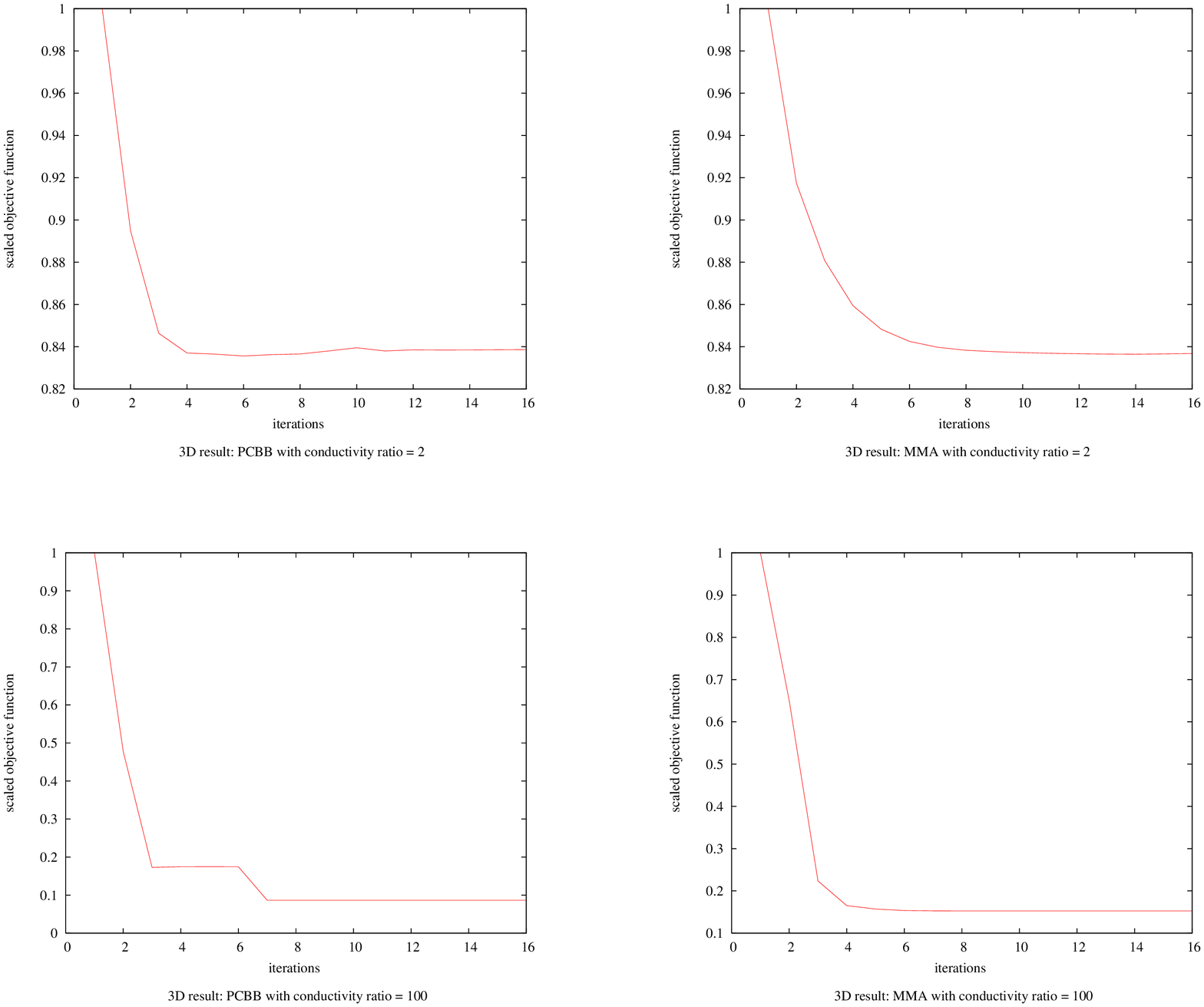}%
\caption{Variations of scaled objective function values during optimization
cycles for 3D examples.
}%
\label{fig:obj_hist_3d}%
\end{center}%
\end{figure}%
\begin{figure}[ht]%
\begin{center}%
\includegraphics[width=14.cm]{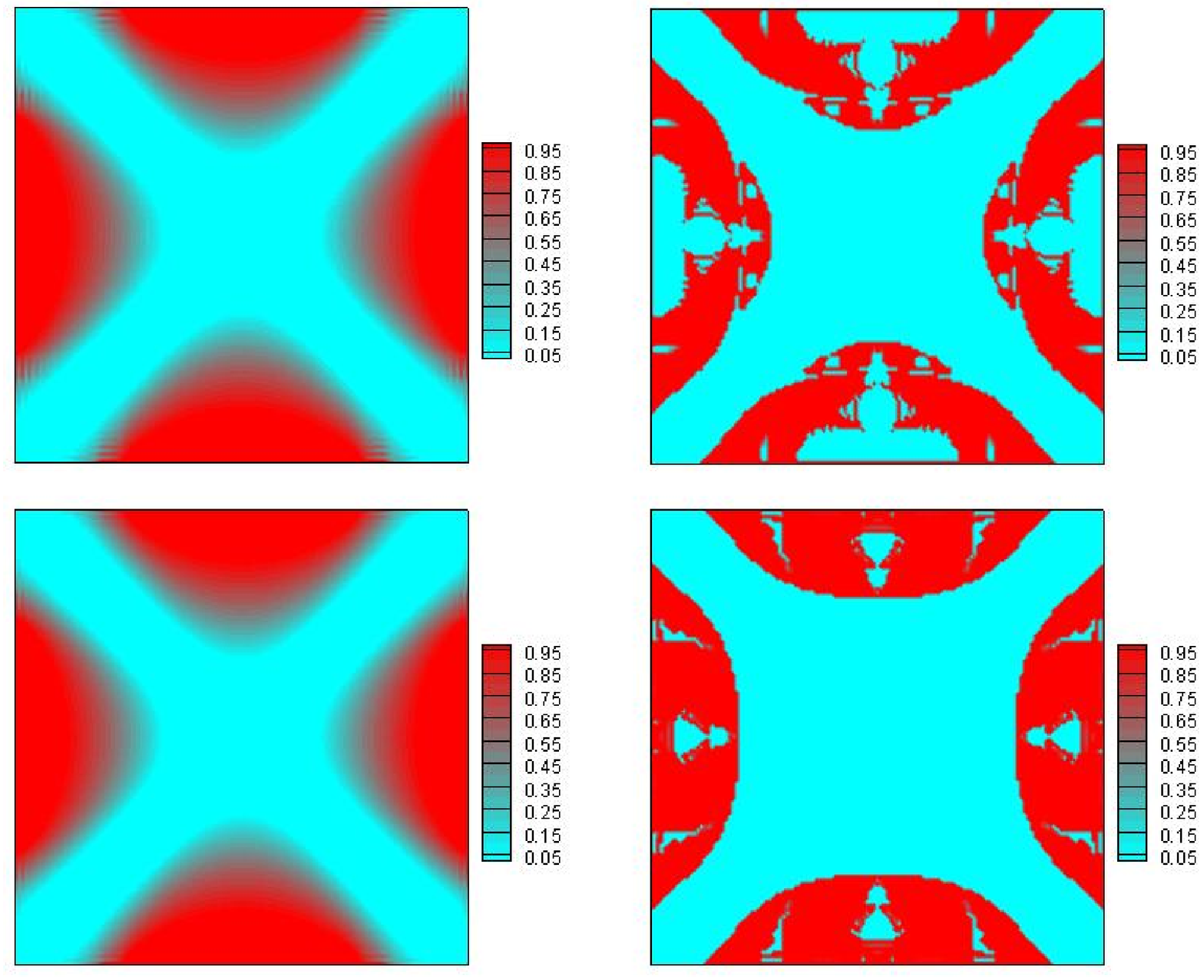}%
\caption{Final distribution of material field, $w$, for 2D examples
related to PCBB (top) and MMA (bottom) for conductivity ratio 2
(left) and 100 (right).
}%
\label{fig:final_top_2d}%
\end{center}%
\end{figure}%
\begin{figure}[ht]%
\begin{center}%
\includegraphics[width=14.cm]{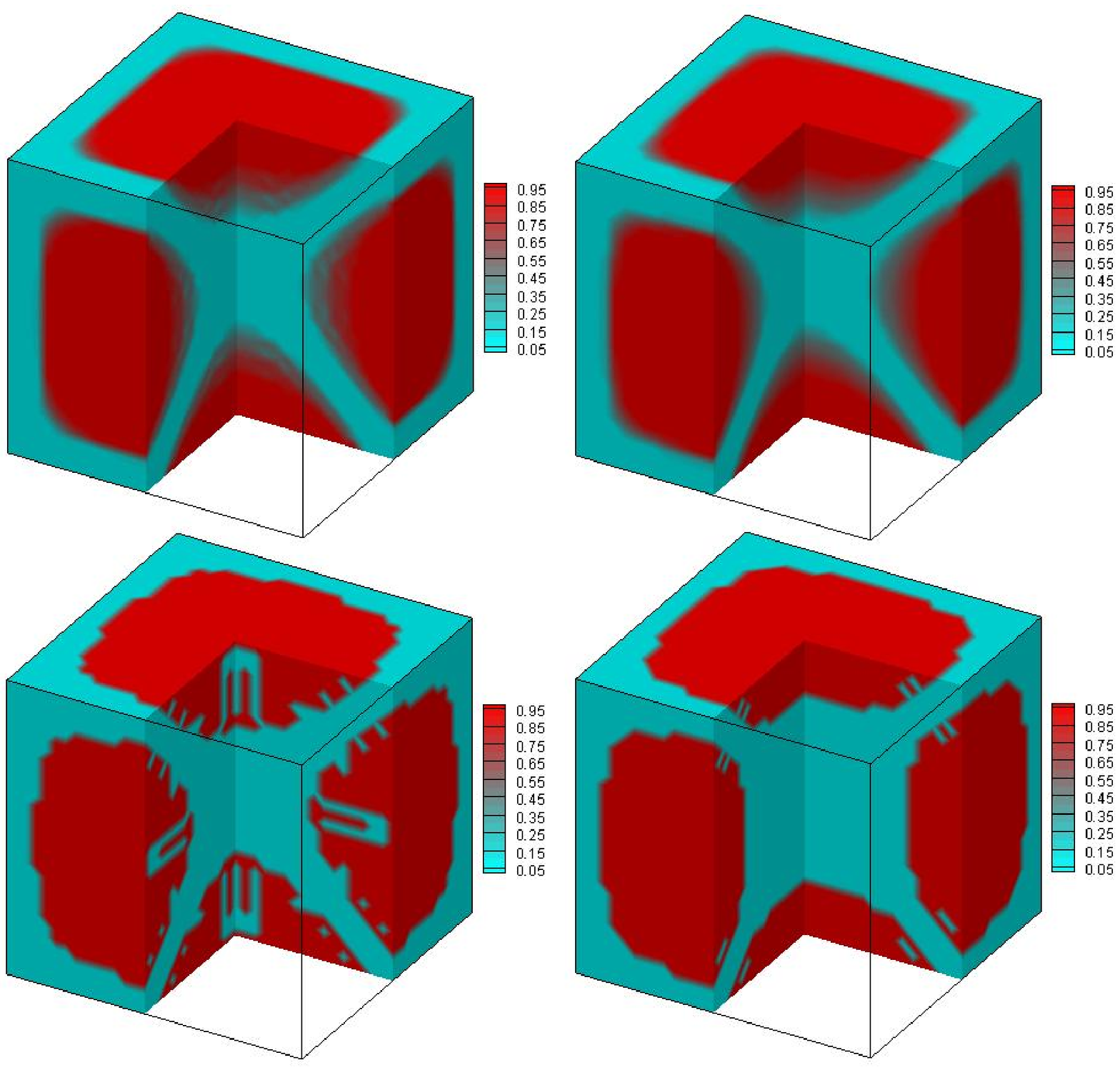}%
\caption{Final distribution of material field, $w$, for 3D examples,
related to PCBB (top) and MMA (bottom) conductivity ratio 2 (left) and 100 (right).
}%
\label{fig:final_top_3d}%
\end{center}%
\end{figure}%
\begin{table}[ht]%
\centering %
\caption{Variations of step size $\alpha_k$ during the fist 10
optimization cycles in PCBB algorithm for 2D examples (the decimals
are dropped in the table). $\alpha(2)$ and $\alpha(100)$ denote
values of $\alpha_k$ for the conductivity ratio 2 and 100
respectively.}
\label{tab:aplha}%
\begin{tabular}{ccccccccccc}%
\\\hline\hline%
iter & 1 & 2 & 3& 4 & 5 & 6 & 7 & 8 & 9 & 10  \\ \hline %
$\alpha(2)$& 37  & 48 & 124& 133 & 52 & 44 & 54 & 147 & 305 & 286   \\ \hline %
$\alpha(100)$ & 53 & 10$^{30}$ &  10$^{30}$ & 243 & 164 & 100 & 104
&
10$^{30}$ & 10$^{30}$ & 1093  \\ %
\hline\hline%
\end{tabular}%
\end{table}%
%


\clearpage

\section{Closing remarks}

Recent studies on the spectral projected gradient methods show
that these class of methods are very promising for solving the
large-scale convex constrained optimization problems when the
projection on the feasible set can be performed efficiently. In
this paper, we presented a particular spectral projected gradient
method called PCBB (Projected Cyclic Barzilai-Borwein) for solving
volume constrained topology optimization problems. This method
applies the cyclic Barzilai-Borwein stepsizes and uses the most
recent adaptive nonmonotone line search techniques, which greatly
improves the efficiency of the method as well as ensures its
global convergence.  By exploring the structure of the admissible
set of the volume constrained topology optimization problem, the
projection step can be performed very efficiently. In addition to
high efficiency, our presented method also enjoys the following
features: easy implementation, minimum memory requirement, no need
for second order (Hessian) information, not sensitive to data
noise, feasibility being strictly maintained during optimization
process. All these features are essential for a successful
numerical method to solve large-scale topology optimization
problems.

Our numerical results indicate our presented methods are very promising
and well-suited for the class of topology optimization problems
considered in this paper. Comparing our results with those of MMA, a
well accepted method in shape and topology optimization community,
it seems the presented methods could be a very competitive
alternative choice for solving the class of topology design problems.
Finally, our results suggest that including spectral step size and
a nonmonotone globalization strategy in MMA, the MMA algorithm could
be further significantly improved. This would be a topic for our continuing research.

\bibliographystyle{plain} 
\bibliography{pcbb}%

\end{document}